\documentclass[twoside,12pt]{cmft}

\usepackage{graphicx}

\pagespan{1}{?}
\cmftinfo{XXYYYZZ}

\newtheorem{theorem}{Theorem}

 \newtheorem{proposition}[theorem]{Proposition}
 
\theoremstyle{remark}
\newtheorem{definition}{Definition}

\DeclareMathOperator{\Res}{\mathcal{R}}
\DeclareMathOperator*{\res}{\mathrm{res}}

\def\R#1{\mathbb{R}^{#1}}

\def\I{\mathrm{i}}

\def\Log{{\rm Log\,}}
\def\D{{\mathbb D}}
\def\R{{\mathbb R}}
\def\C{{\mathbb C}}

\def\P{{\mathbb P}}

\begin{document}

\title{On the exponential transform of multi-sheeted algebraic domains}
\date{December 30, 2010}

\author[B. Gustafsson]{Bj\"orn Gustafsson}
\email{gbjorn@kth.se}
\address{Mathematical department, KTH, 10044 Stockholm}

\author[V. Tkachev]{Vladimir G. Tkachev}
\email{tkatchev@kth.se}
\address{Mathematical department, KTH, 10044 Stockholm}


\thanks{This research
has been supported by the Swedish Research Council, and is part of
the European Science Foundation Networking programme ``Harmonic and
Complex Analysis and Applications HCAA''}

\begin{abstract}
We introduce multi-sheeted versions of algebraic domains and
quadrature domains, allowing them to be branched covering surfaces
over the Riemann sphere. The two classes of domains turn out to be
the same, and the main result states that the extended exponential
transform of such a domain agrees, apart from some simple factors,
with the extended elimination function for a generating pair of
functions. In an example we discuss the algebraic curves associated
to level curves of the Neumann oval, and determine which of these
give rise to multi-sheeted algebraic domains.
\end{abstract}

\maketitle


\section{Introduction}

Extending some previous work \cite{Putinar94}, \cite{Putinar96},
\cite{Putinar98}, \cite{Gustafsson-Tkachev09},
\cite{Gustafsson-Tkachev10} on the rationality of the exponential
transform we here go on to consider what we believe is the most
general type of domains for which the exponential transform, in an
extended sense, can be expected to have a ``core'' consisting of a
rational function. Around this core there will then be
``satellites'' of some rather trivial factors, depending on the
regimes at the locations of the independent variables.

The ``domains'' we consider will actually be covering surfaces over
the Riemann sphere. The terminology ``quadrature Riemann surface''
has already been introduced by M.~Sakai \cite{Sakai88} for the type
of domains in question. An alternative name could be ``multi-sheeted
algebraic domain'', to extend a terminology used by A.~Varchenko and
P.~Etingof \cite{Varchenko-Etingof94}. In the present work we shall
mostly use the latter terminology because we will not emphasize so
much the quadrature properties, but rather take as a starting point
the way the domains are generated, namely by a pair of meromorphic
function on a compact Riemann surface.

Given a bounded domain $\Omega\subset \mathbb{C}$, the traditional
exponential transform \cite{Carey-Pincus74}, \cite{Putinar94},
\cite{Putinar96}, \cite{Putinar98}, \cite{Gustafsson-Putinar98} of
$\Omega$ is the function of two complex variables defined by
\begin{equation*}
E_\Omega (z,w)=\exp[\frac{1}{2\pi\I}\int_\Omega
\frac{d\zeta}{\zeta -z}\wedge \frac{d\bar{\zeta}}{\bar{\zeta} -\bar {w}}] \quad (z,w\in \C).
\end{equation*}
For the unit disk $\mathbb{D}$ it is (see
\cite{Gustafsson-Putinar98}):
\begin{equation}\label{et3}
E(z,w)=
\begin{cases}
1-\frac{1}{z\bar w} \quad &(z,w\in \mathbb{C}\setminus \overline{\mathbb{D}}),\\
1-\frac{\bar z}{\bar w}   \quad  &(z\in\D,\, w\in \mathbb{C}\setminus \overline{\mathbb{D}}),\\
1-\frac{w}{z} \quad &(z\in \mathbb{C}\setminus \overline{\mathbb{D}},\,w\in \D,\\
\frac{|z-w|^2}{1-z\bar w} \quad &(z,w\in \mathbb{D}).\\
\end{cases}
\end{equation}
Using the Schwarz function \cite{Davis74}, \cite{Shapiro92} for
$\partial\D$,
$$
S(z)=\frac{1}{z},
$$
(\ref{et3}) can be written
\begin{equation}\label{expelimD}
E(z,w)=(\frac{\bar z -\bar w}{S(z)-\bar w})^{\rho(z)} (\frac{ z -
w}{z-\overline{S(w)}})^{\rho(w)} \mathcal{E} (z,\bar w),
\end{equation}
where $\mathcal{E} (z,\bar w)$ is the rational function
$$
\mathcal{E} (z,\bar w)=\frac{z\bar w-1}{z\bar w}
$$
and $\rho=\chi_{\D}$ is the characteristic function of $\D$. The
expression (\ref{expelimD}) reveals the general structure of the exponential
transform of any algebraic domain.

The function $\mathcal{E} (z,\bar w)$ is an instance of the
elimination function, which can be defined by means of the
meromorphic resultant $\Res (f,g)$ of two meromorphic functions $f$
and $g$. The resultant is defined as the multiplicative action of
$g$ on the divisor $(f)$ of $f$,  namely $\Res (f,g)=g((f))$, and
the elimination function is $\mathcal{E}_{f,g} (z, \bar w) =\Res (f
-z, g-\bar w)$. In the case of the unit disk, or any quadrature
domain, the relevant elimination function which enters into the
exponential transform is the one with $f(\zeta)=\zeta$,
$g(\zeta)=S(\zeta)$. This means that we in the present example get
$$
\mathcal{E} (z,\bar w) =\Res (\zeta -z,S(\zeta)-\bar w)
=(S(\zeta)-\bar w)((\zeta-z))
$$
$$
=(S(\zeta)- \bar w)(1\cdot (z) - 1\cdot (\infty))
=\frac{S(z)-\bar w}{S(\infty)-\bar w}= \frac{z\bar w-1}{z\bar w},
$$
as desired.

The aim of the present paper is to generalize the formula (\ref{expelimD}) as far as
possible. This will involve an extended exponential transform in four complex variables and
an analogous extended elimination function in four variables, defined in terms of a conjugate pair of
meromorphic functions on a fairly general compact symmetric Riemann surface.

The paper is organized as follows. Sections~\ref{sec:Cauchy} and
\ref{sec:resultant} contain general preliminary material. In
Section~\ref{sec:multi} we introduce the concepts of multi-sheeted
algebraic domains and quadrature Riemann surfaces and prove that
they are equivalent. The main result is stated in
Section~\ref{sec:main} and proved in Section~\ref{sec:proof}.
Section~\ref{sec:example}, finally, is devoted to examples, namely
the ellipse and Neumann's oval.


\section{The Cauchy and exponential transforms}\label{sec:Cauchy}

The Cauchy transform of a bounded density function $\rho$ in $\C$ is
$$
C_\rho (z)=\frac{1}{2\pi\I}\int \frac{\rho (\zeta)d\zeta \wedge d\bar \zeta}{\zeta-z}.
$$
Typically the functions $\rho$ which will appear in this paper will be like the characteristic function of a domain, or
the corresponding integer valued counting function for a multi-sheeted domain. From
$C_\rho$, the density $\rho$ can be recovered by
\begin{equation}\label{dbarCauchy}
\rho(z)=\frac{\partial C_\rho (z)}{\partial\bar z},
\end{equation}
to be interpreted in the sense of distributions.

If one writes the definition of the Cauchy transform as
$$
C_\rho (z)=\frac{1}{2\pi\I}\int \rho(\zeta)\frac{d\zeta }{\zeta-z}
\wedge d\bar \zeta
$$
one realizes that it suffers from a certain lack of symmetry. A more
balanced object would be the ``double Cauchy transform'',
\begin{equation}\label{doubleCauchy}
C_\rho (z,w)=\frac{1}{2\pi\I}\int \rho(\zeta)\frac{d\zeta }{\zeta-z} \wedge \frac{d\bar \zeta}{\bar \zeta-\bar w}.
\end{equation}
In fact, this double transform is much richer than the original
transform, and after exponentiation it gives the by now quite well
studied \cite{Carey-Pincus74}, \cite{Putinar94}, \cite{Putinar96},
\cite{Putinar98}, \cite{Gustafsson-Putinar98} (etc) exponential
transform:
$$
E_\rho(z,w)=\exp {C_\rho (z,w)}.
$$
The original Cauchy transform can be recovered as
$$
C_\rho(z)=\res_{w=\infty} C_\rho (z,w) =-\lim_{w\to \infty} {\bar w} C_\rho(z,w),
$$
at least if $\rho$ vanishes in a neighborhood of infinity. One disadvantage with the double Cauchy transform is that the formula (\ref{dbarCauchy})
turns into the more complicated
\begin{equation}\label{dbarCauchy1}
\frac{\partial C_\rho (z,w)}{\partial \bar z}=\frac{\rho(z)}{\bar z-\bar w}.
\end{equation}
On the other hand we have the somewhat nicer looking
$$
\frac{\partial^2 C_\rho (z,w)}{\partial \bar z\partial w}=-\pi \rho(z)\delta(z-w),
$$
where $\delta$ denotes the Dirac distribution.

Now, even the double Cauchy transform is not entirely complete. It
contains the Cauchy kernel $\frac{d\zeta}{\zeta-z}$, which is a
meromorphic differential on the Riemann sphere with a pole at
$\zeta=z$, but it has also a pole at $\zeta=\infty$. It is natural
to make the latter pole visible and movable. That would have the
additional advantage that one can avoid the two Cauchy kernels which
appear in the definitions of the double Cauchy transform and the
exponential transform to have coinciding poles (namely at infinity).
Thus we arrive naturally at the extended Cauchy and exponential
transforms:
\begin{equation}\label{extendedCauchy}
C_\rho(z,w;a,b)=\frac{1}{2\pi\I}\int \rho(\zeta)(\frac{d\zeta }{\zeta-z} -\frac{d\zeta }{\zeta-a})
\wedge (\frac{d\bar \zeta}{\bar \zeta-\bar w}-\frac{d\bar \zeta}{\bar \zeta-\bar b}),
\end{equation}
\begin{equation}\label{extendedexp}
E_\rho(z,w;a,b)=\exp { C_\rho(z,w;a,b)}=\frac{E_\rho(z,w)E_\rho(a,b)}{E_\rho(z,b)E_\rho(a,w)}.
\end{equation}
If the points $z$, $w$, $a$, $b$ are taken to be all distinct, then
both transforms are well defined and finite for any bounded density
function $\rho$ on the Riemann sphere. For example, with
$\rho\equiv 1$, $E_\rho(z,w;a,b)$ turns out to be the modulus
squared of the cross-ratio. See \cite{Gustafsson-Tkachev09} for
further details.


\section{The resultant and the elimination function}\label{sec:resultant}

Here we shall briefly review the definitions of the meromorphic resultant and the
elimination function, as introduced in \cite{Gustafsson-Tkachev09},  referring to that paper
for any details.
If $f$ is a meromorphic function on any compact Riemann surface $M$ we denote by $(f)$ its divisor of zeros and poles,
symbolically $(f)=f^{-1}(0)-f^{-1}(\infty)$. If $D$ is any divisor and $g$ is a meromorphic function
we denote by $g(D)$ the multiplicative action of $g$ on $D$. For example, if $D=1\cdot (a)+1\cdot (b)-2\cdot (c)$,
$a,b,c\in M$, then $g(D)=\frac{g(a)g(b)}{g(c)^2}$. Now the meromorphic resultant between $f$ and $g$ is, by definition,
$$
\Res (f,g)=g((f))
$$
whenever this makes sense.

The elimination function is
$$
\mathcal{E}_{f,g}(z,w)=\Res(f-z,g-w),
$$
where $z,w\in\C$ are parameters. It is always a rational function in $z$ and $w$,
more precisely of the form
\begin{equation}\label{QPR}
\mathcal{E}_{f,g}(z,w)=\frac{Q(z,w)}{P(z)R(w)},
\end{equation}
where $Q$, $P$ and $R$ are polynomials,
and it embodies the necessary (since $M$
is compact) polynomial relationship between $f$ and $g$:
$$
\mathcal{E}_{f,g}(f(\zeta),g(\zeta))=0 \quad (\zeta\in M).
$$
We also have the extended elimination function, defined by
$$
\mathcal{E}_{f,g}(z,w; a, b)=\Res(\frac{f-z}{f-a},\frac{g-w}{g-b}).
$$

To relate the elimination function to the exponential transform one
needs integral formulas for the elimination function. If $f$ is
meromorphic on $M$ with divisor $(f)$, let $\sigma_f$ be a $1$-chain
such that $\partial \sigma_f=(f)$ and such that $\log f$ has a
single-valued branch, which we denote $\Log f$, in $M\setminus{\rm
supp\,} \sigma_f$. Then $\Log f$ can be viewed as a distribution on
$M$, and its exterior differential in the sense of distributions (or
currents) is
\begin{equation}\label{dLogf}
d\Log f = \frac{df}{f} -2\pi\I dH_{\sigma_f}.
\end{equation}
Here $dH_{\sigma_f}$ is the $1$-form current supported by $\sigma_f$ and defined locally, away from $\partial \sigma_f$,
as the differential (in the sense of currents) of that function $H_{\sigma_f}$ which is $+1$ on the right-hand side of
$\sigma_f$, zero on the left-hand side. Globally $dH_{\sigma_f}$ is not exact (despite the notation), not even closed.
To be precise,
$$
d(dH_{\sigma_f})=\frac{1}{2\pi\I}d(\frac{df}{f})=\delta_{(f)}dx\wedge dy,
$$
where $\delta_{(f)}$ denotes the finite distribution of point masses
(or charges) corresponding to $(f)$. We shall also need the fact
that $dH_{\sigma_f}$ has the period reproducing property
\begin{equation}\label{inttau}
\int_M dH_{\sigma_f}\wedge \tau =\int_{\sigma_f} \tau,
\end{equation}
holding for any smooth $1$-form $\tau$.

Now we have (essentially Theorem~2 in \cite{{Gustafsson-Tkachev09}})
\begin{equation}\label{elimint}
\mathcal{E}_{f,g}(z,w;a,b)= \exp[\frac{1}{2\pi \I}\int_M
(\frac{df}{f-z}-\frac{df}{f-a})\wedge d\, \Log \frac{g-w}{g-b}].
\end{equation}
The integrand is a $2$-form current with support on the $1$-chains
$\sigma_{g-w}$ and $\sigma_{g-b}$ (because away from these curves
the integrand contains $d\zeta\wedge d\zeta$), so the integral is
rather a line integral than an area integral. In fact, the above can
also be written
$$
\mathcal{E}_{f,g}(z,w;a,b)= \exp[\int_{\sigma_{g-w}} (\frac{df}{f-z}-\frac{df}{f-a})-\int_{\sigma_{g-b}} (\frac{df}{f-z}-\frac{df}{f-a})],
$$
which perhaps clarifies the connection to the definition of the elimination function.


\section{Multi-sheeted algebraic domains}\label{sec:multi}

The boundary of a quadrature domain (algebraic domain) is an
algebraic curve, but by no means every algebraic curve arises in
this way. However, the gap between the two classes of objects can be
reduced considerably by extending the notion of a quadrature domain,
allowing it to have several sheets and to be branched over the
Riemann sphere. This will take essentially one half of all algebraic
curves into the framework of quadrature domains and exponential
transforms. One  step in this direction was taken in Sakai
\cite{{Sakai88}}, where a notion of quadrature Riemann surface was
introduced in a special case. Below we shall take some further
steps.

Let $M$ be any symmetric compact (closed) Riemann surface. Slightly
more generally, we shall allow $M$ to be disconnected, namely to be
a finite disjoint union of Riemann surfaces. The symmetry means that
$M$ is provided with an anticonformal involution $J:M\to M$, $J\circ
J={\rm identity}$. If $M$ is disconnected then $J$ is allowed to map
one component of $M$ onto another. Let $\Gamma$ denote the set of
fixed points of $J$. Simple examples of symmetric Riemann surfaces
are $M=\mathbb{P}$ (the Riemann sphere) with the involution being
either $J_1(\zeta)=1/\bar{\zeta}$ or $J_2(\zeta)=-1/\bar{\zeta}$. In
the first case $\Gamma=\{\zeta: |\zeta|=1\}$ and $M\setminus\Gamma$
has two components, in the second case $\Gamma$ is empty and hence
$M\setminus\Gamma$ has only one component. One can also think of
identifying the points $\zeta$ and $J(\zeta)$. The identification
spaces become, in the first case ($J=J_1$) the unit disk together
with its boundary, and in the second case ($J=J_2$) the projective
plane, thus a nonorientable surface.

In general, the orbit space $N=M/J$, obtained  by identifying
$\zeta$ and $J(\zeta)$ for any $\zeta\in M$, is a Klein surface,
possibly with boundary. A Klein surface \cite{Alling-Greenleaf71} is
defined in the same way as a Riemann surface except that it is
allowed to be nonorientable and that both holomorphic and
antiholomorphic transition functions between coordinates are
allowed. The possible boundary points of $N$ are those coming from
$\Gamma$ under the identification. From the Klein surface $N$, $M$
can be recovered by a natural doubling procedure, described in
Section~2.2 of \cite{Schiffer-Spencer54} and in
\cite{Alling-Greenleaf71}, for example. The latter reference
actually describes several types of doubles (the complex double, the
orienting double and the Schottky double), but the description in
\cite{Schiffer-Spencer54} will be enough for our purposes. In case
$N$ is orientable and has a boundary the doubling procedure gives
what is usually called the Schottky double, named after the inventor
of the idea, F.~Schottky  \cite{Schottky77}. The idea was later
extended to more general surfaces by F. Klein \cite{Klein91}.

Continuing the discussion of $M$, $J$ and $N$, if $M$ is connected
but $M\setminus \Gamma$ disconnected, then $M\setminus \Gamma$ has
exactly two components, say $M_+$ and $M_-$, and $J$ maps each of
them onto the other. It then follows that  $(M\setminus\Gamma)/J$
can be identified with $M_+$ (or $M_-$), and in particular that $N$
is orientable, hence is (after choice of orientation) an ordinary
Riemann surface with boundary. If $M\setminus \Gamma$ is connected
then $N$ necessarily is nonorientable.

It is relevant to allow $M$ to have several components. For example,
in Section~\ref{sec:example} we will encounter the double of the
Riemann sphere, which simply is two Riemann spheres with the
opposite conformal structure and with $J$ mapping one onto the
other.

Now to the definition of ``multi-sheeted algebraic domain''. There
are two ingredients. The first is a compact symmetric Riemann
surface $(M,J)$ such that $M$ is connected and such that
$M\setminus\Gamma$ has two components, one of which, call it $M_+$,
is to be selected. We could equally well have started with $M_+$, to
be any Riemann surface with boundary (bordered Riemann surface), and
then let $M$ be the double of $M_+$. The second ingredient is a
nonconstant meromorphic function $f$ on $M$.

\begin{definition}\label{def:multi}
A  {\em multi-sheeted algebraic domain} is (represented by) a pair
$(M_+,f)$, where $M_+$ is a bordered Riemann surface and $f$ is a
nonconstant meromorphic function on the double $M$ of $M_+$. Two
pairs, $(M_+,f)$ and $(\tilde{M}_+, \tilde{f})$ are considered the
same if there is a biholomorphic mapping $\phi:M\to \tilde{M}$ such
that $\phi\circ J=\tilde{J}\circ \phi$ and $f=\tilde{f}\circ \phi$.
\end{definition}

The equivalence simply means that it is the image $f(M_+)$, with
appropriate multplicities, which counts. Note that also $(M_-,f)$ is
a multi-sheeted algebraic domain, if $(M_+,f)$ is. Trivial examples
of a multi-sheeted algebraic domain are obtained by taking $M=\P$
and $J(\zeta)=1/\bar{\zeta}$. Then with $f$ any nonconstant rational
function $(\D,f)$ will be a multi-sheeted algebraic domain, as well
as $(\P\setminus\overline{\D},f)$. Some further examples will be
discussed in Section~\ref{sec:example}.

Along with $f$, meromorphic on $M$, the symmetry $J$ provides one
more meromorphic function on $M$, namely
$$
f^*=\overline{(f\circ J)}.
$$
With $Q(z,w)$ the polynomial in (\ref{QPR}) for $g=f^*$, the map
$M\ni \zeta\mapsto(f(\zeta),f^*(\zeta))$ parametrizes the curve
$Q(z,\bar w)=0$ (or, better, its projective counterpart). This
parametrization is one-to-one (except for finitely many points) if
and only if $f$ and $f^*$ generate the field of meromorphic
functions on $M$ (form a primitive pair in the terminology of
\cite{Farkas-Kra80}). This will generically be the case if $f$ is
chosen ``at random'', but there are certainly many counterexamples.
For example, $f$ may be already symmetric in itself, i.e., $f=f^*$,
and then $f$ and $f^*$ is a primitive pair only if $M$ is the
Riemann sphere and $f$ is a M\"obius transformation. If $f$ and
$f^*$ are not a primitive pair, then the polynomial $Q(z,w)$ in
(\ref{QPR}) is reducible.

In case $M_+$ is planar (i.e., is topologically equivalent to a
planar domain) and $f$ is univalent on $M_+$ without poles on
$M_+\cup\Gamma$, then $\Omega=f(M_+)$ is an ordinary algebraic
domain, in other words a classical quadrature domain. In this case
$f$ and $f^*$ do form a primitive pair (see \cite{Gustafsson83}).
The reference to ``quadrature'' can be explained in terms of a
residue calculation. In the present generality it is natural to use
the spherical metric on the Riemann sphere in place of the customary
Euclidean metric. This will remove some integrability problems, and
the point of infinity can be treated on the same footing as other
points. Quadrature domains for the spherical measure have been
previously discussed, at least in fluid dynamic contexts, for
example in \cite{Varchenko-Etingof94} (Hele-Shaw flow),
\cite{Crowdy-Cloke03} (vortex patches).

Let $h$ be a function holomorphic in a neighborhood of
$M_+\cup\Gamma$ and assume, for simplicity, that $f$ has no poles on
$\partial M_+$ (such poles will actually cause no problems anyway).
At least away from poles of $f$ we have, by exterior
differentiation,
$$
d\left(\frac{h\bar{f}
df}{1+f\bar{f}}\right)=
\frac{hd\bar{f}\wedge
df}{(1+f\bar{f})^2}.
$$
Therefore, if $f$ has poles of orders $n_j$ at $a_j\in M_+$ then
$$
\frac{1}{2\pi\I}\int_{M_+} \frac{hd\bar{f}\wedge
df}{(1+f\bar{f})^2}
=\lim_{\varepsilon\to 0}
\frac{1}{2\pi\I}\int_{M_+\cap\{|f|<1/\varepsilon\}} \frac{hd\bar{f}\wedge
df}{(1+f\bar{f})^2}=
$$
$$
=\frac{1}{2\pi\I}\int_{\partial M_+} \frac{h\bar{f} df}{1+f\bar{f}}
-\lim_{\varepsilon\to 0}\frac{1}{2\pi\I}\oint_{|f|=1/\varepsilon} \frac{h\bar{f} df}{1+f\bar{f}}=
$$
$$
=\frac{1}{2\pi\I}\int_{\partial M_+} \frac{h{f^*} df}{1+f f^*}+ \sum_j n_j h(a_j)
=\sum_{M_+}\res \frac{h{f^*} df}{1+f f^*}+ \sum_j n_j h(a_j) .
$$
Here it turns out that the terms $\sum n_j h(a_j)$ cancel with corresponding terms (with negative sign) in the residue contribution
unless $f^*$ happens to have zeros at the points $a_j$. In any case, with or without
cancellations, the right member above is of the form
$$
L(h)=\sum_{k=1}^m\sum_{j=0}^{n_k-1} c_{kj}h^{(j)}(b_k),
$$
i.e., equals the action on $h$ by a distribution with support in
finitely many points. The quadrature nodes $b_k$ are the poles of
$\frac{f^* df}{1+ff^*}$.

If $f$ is univalent the above identity becomes an ordinary
quadrature identity (although for the spherical measure) of the form
\begin{equation}\label{quadratureidentity}
\frac{1}{2\pi\I}\int_{\Omega} g(z) \frac{d\bar{z}\wedge
dz}{(1+|z|^2)^2} =\tilde{L}(g),
\end{equation}
holding for all integrable analytic functions $g$ in $\Omega=f(M_+)$. Here
the right member is given by $\tilde{L}(g)=L(g\circ f)$,
which still is the action of a distribution with finite support.

When $f$ is not univalent one should still think of the quadrature
identity in the same way as in (\ref{quadratureidentity}), just with
the difference that $\Omega$ is a region with several sheets over
the Riemann sphere. The test functions should be allowed to take
different values at points lying above one and the same point, but
on different sheets. Thus those of the form $g\circ f$ (i.e., those
which would become $g(z)$ in a formulation like
(\ref{quadratureidentity})) are too special. This is most easily
expressed by pulling everything back to $M_+$, in which case we
simply have the originally obtained identity
\begin{equation}\label{qi}
\frac{1}{2\pi\I}\int_{M_+} h\frac{d\bar{f}\wedge
df}{(1+|f|^2)^2} =L(h),
\end{equation}
for 
$h$ holomorphic in a neighborhood of $M_+\cup\Gamma$, and by
approximation for functions $h$ holomorphic and integrable (with
respect to $\frac{1}{2\pi\I}\frac{d\bar f\wedge df}{(1+|f|^2)^2}$)
in $M_+$. An equivalent formulation is that there exists a positive
divisor $D$ in $M_+$ such that
\begin{equation}\label{qiD}
\frac{1}{2\pi\I}\int_{M_+} h\frac{d\bar{f}\wedge
df}{(1+|f|^2)^2} =0
\end{equation}
holds for every holomorphic and integrable function $h$ in $M_+$
with $(h)\geq D$.

\begin{definition}
A  {\em quadrature Riemann surface (for the spherical metric)} is a
pair $(M_+,f)$, where $M_+$ is a bordered Riemann surface and $f$ is
a nonconstant meromorphic function on $M_+$ such that
$\frac{1}{2\pi\I}\int_{M_+} \frac{d\bar{f}\wedge
df}{(1+|f|^2)^2}<\infty$ and such that (\ref{qi}) (or (\ref{qiD}))
holds for some $L$ (respectively $D$) and the indicated classes of
functions $h$.
\end{definition}

The notion of equivalence between pairs is the same as in
Definition~\ref{def:multi}. If $f$ is meromorphic on $M$ then we
have
$$
\frac{1}{2\pi\I}\int_{M_+} \frac{d\bar{f}\wedge df}{(1+|f|^2)^2}
\leq \frac{1}{2\pi\I}\int_{M} \frac{d\bar{f}\wedge df}{(1+|f|^2)^2}
=\frac{n}{2\pi\I}\int_{\P} \frac{d\bar{z}\wedge
dz}{(1+|z|^2)^2}=n<\infty,
$$
where $n$ is the order of $f$. Thus we have proved one direction of
the following.

\begin{proposition}
A pair $(M_+,f)$ is a multi-sheeted algebraic domain if and only if it is a quadrature Riemann surface.
\end{proposition}

\begin{proof}
It remains to prove that a quadrature Riemann surface is a
multi-sheeted algebraic domain. In many special cases this has
already been done (for the Euclidean metric), see for example
\cite{Aharonov-Shapiro76}, \cite{Gustafsson83}, \cite{Sakai88}. We
shall discuss here the general case under the simplifying assumption
that $f$ is meromorphic in a neighborhood of $M_+\cup \Gamma$.

It will be convenient to introduce some further notation. With
$A\subset M$ any subset and $D$ any divisor on $M$ we denote by
$\mathcal{O}_D (A)$ the set of functions $h$ meromorphic in a
neighborhood of $A$ and with $(h)\geq D$. Similarly,
$\mathcal{O}_D^{1,0} (A)$ denotes the set of $1$-forms $\omega$,
meromorphic in a neighborhood of $A$ and satisfying $(\omega)\geq
D$. We shall also use standard notations for cohomology groups.

Now, returning to the previous residue calculation one realizes that
what needs to be proven is that if
\begin{equation}\label{qibdryD}
\frac{1}{2\pi\I}\int_{\partial M_+} h\frac{\bar{f} df}{1+f \bar f} =0
\end{equation}
holds for all $h\in\mathcal{O}_D(M_+\cup\Gamma)$, for some
sufficiently large divisor $D$, then $f$ extends to a meromorphic
function on $M$. On $\Gamma$ we have
$$
\frac{\bar{f} df}{1+f \bar f}=\frac{{f^*} df}{1+f  f^*},
$$
where the right member is holomorphic in a neighborhood of $\Gamma$
and can be viewed as representing an element in the cohomology group
$H^1(M,\mathcal{O}_{-D}^{1,0})$. When $D$ is strictly positive this
group is trivial since, by Serre duality \cite{Serre55},
\cite{Forster77},  $H^1(M,\mathcal{O}_{-D}^{1,0})\cong
H^0(M,\mathcal{O}_{D})^*=\mathcal{O}_{D}(M)^*=0$ (here star $^*$
denotes dual space), hence there exist
$\omega_{\pm}\in\mathcal{O}_{-D}^{1,0}(M_{\pm}\cup\Gamma)$ such that
$$
\frac{{f^*} df}{1+f  f^*}=\omega_+ -\omega_-
$$
in a neighborhood of $\Gamma$.

Clearly
$$
\frac{1}{2\pi\I}\int_{\partial M_+} h\omega_+ =0
$$
for  $h\in\mathcal{O}_D(M_+\cup\Gamma)$ since the integrand is
holomorphic in $M_+$, hence (\ref{qibdryD}) reduces to the statement
that
\begin{equation}\label{duality}
\frac{1}{2\pi\I}\int_{\partial M_+} h\omega_- =0
\end{equation}
for  $h\in\mathcal{O}_D(M_+\cup\Gamma)$. At this point we may use a
general duality theorem (also related to Serre duality) going back
to the work of J.~Silva \cite{Silva50} and G.~K\"othe
\cite{Kothe53}, extended by A.~Grothendieck \cite{Grothendieck53}
and, in the form we need it, C.~Auderset \cite{Auderset80}. It
states that the bilinear form in $h$ and $\omega_-$ defined by the
left member of (\ref{duality}) induces (with the path of integration
moved slightly into $M_-$) a non-degenerate pairing
$$
\mathcal{O}_D (M_+\cup\Gamma)/\mathcal{O}_D
(M)\times\mathcal{O}_{-D}^{1,0}({M_-})/\mathcal{O}_{-D}^{1,0}({M})\to\C,
$$
which exhibits each of the quotient spaces as the dual space of the
other.

In view of this (\ref{duality}) implies that
$\omega_-\in\mathcal{O}_{-D}^{1,0}({M})$, in particular that
\begin{equation}\label{omega}
\frac{{f^*} df}{1+f  f^*}=\omega_+
-\omega_-\in\mathcal{O}_{-D}^{1,0}({M_+}\cup \Gamma).
\end{equation}
Since $f$ is meromorphic in a neighborhood of $M_+\cup\Gamma$,
(\ref{omega}) implies that also $f^*$ is meromorphic there, hence
that $f$ actually is meromorphic on all $M$, as was to be proved.

The usage of the general duality theorems above can be replaced by more direct arguments, like applying
(\ref{qi}) to suitable Cauchy kernels. Specifically we may choose, with $\zeta\in M_-$,
$$
h(z)= \Phi (z,\zeta;z_0,\zeta_0) d\zeta,
$$
where the right member is a kernel which in the case of the Riemann
sphere is the usual Cauchy kernel
$$
\Phi (z,\zeta;z_0,\zeta_0) d\zeta=\frac{d\zeta}{\zeta-z}-\frac{d\zeta}{\zeta-z_0},
$$
and which has counterparts with good enough properties on all
compact Riemann surfaces (see \cite{Rodin88}). The point $\zeta_0$
is needed in higher genus. Actually the duality theorem discussed
above can be proved using this kernel.

\end{proof}


\section{Statement of the main result}\label{sec:main}

Let $M$, $M_+$, $\Gamma=\partial M_+$, $f$ be as in
Section~\ref{sec:multi}, more precisely such that $(M_+,f)$ is a
multi-sheeted algebraic domain. To account for the multiplicities of
$f(M_+)$ as a covering of the Riemann sphere we introduce the
integer-valued counting function, or mapping degree,
$$
\rho(z)={\rm card\,}\{\zeta\in M_+: f(\zeta)=z\},
$$
pointwise well-defined for $z\in \P\setminus f(\Gamma)$. It is
understood that points $\zeta$ are counted with the appropriate
multiplicities. Set also
\begin{equation}\label{schwarz}
S=\overline{f\circ J \circ f^{-1}} = f^*\circ f^{-1}.
\end{equation}
This is a multi-valued algebraic function in the complex plane which
contains all local Schwarz functions of $f(\Gamma)$, because for
$z\in f(\Gamma)$ one of the values of $S(z)$ is $\bar{z}$.

We will have to make expressions like $(S(z)-\bar w))^{\rho(z)}$
well-defined, i.e., single-valued, despite $S(z)$ itself being
multi-valued. The natural definition is the following:
$$
(S(z)-\bar w)^{\rho(z)}=(f^*-\bar w)((f-z)|_{M_+}).
$$
Here $(f-z)|_{M_+}$ denotes the restriction of the divisor $(f-z)$
to $M_+$ and the right member then is the multiplicative action of
$f^*-\bar w$ on $(f-z)|_{M_+}$. To spell it out, let
$$
f^{-1}(z)\cap M_+=\{\zeta_1,\dots,\zeta_{\rho(z)}\},
$$
with repetitions according to multiplicities. Then
\begin{equation}\label{schwarzrho}
(S(z)-\bar w)^{\rho(z)}=(f^*(\zeta_1)-\bar
w)\cdot\dots\cdot(f^*(\zeta_{\rho(z)})-\bar w),
\end{equation}
which is a natural definition in view of (\ref{schwarz}).
Clearly $(S(z)-\bar w)^{\rho(z)}$ is an analytic function of $z$ in regions where
$\rho(z)$ is constant.

Now, for the main result, we have two functions which we want to
relate to each other: one is the weighted exponential transform
\begin{equation*}
E_\rho (z,w;a,b) =\exp[\frac{1}{2\pi\I}\int_\P \rho(\zeta)
\biggl(\frac{d\zeta}{\zeta -z}-\frac{d\zeta}{\zeta -a}\biggr)\wedge
\biggl(\frac{d\bar\zeta}{\bar\zeta -\bar
w}-\frac{d\bar\zeta}{\bar\zeta - \bar b}\biggr)],
\end{equation*}
$$
=\exp[\frac{1}{2\pi\I}\int_{M_+}
\biggl(\frac{df}{f -z}-\frac{df}{f -a}\biggr)\wedge
\biggl(\frac{d\bar f}{\bar f -\bar
w}-\frac{d\bar f}{\bar f - \bar b}\biggr)],
$$
which can be viewed as a kind of potential of $\rho$, and the other
is the elimination function,
which is defined by algebraic means and always is a rational
function, namely of the form
\begin{equation*}
\mathcal{E}_{f,f^*} (z, \bar w; a,\bar b)
=\Res(\frac{f-z}{f-a},\frac{f^*-\bar w}{f^*-\bar b})=\frac{Q(z,\bar
w)Q(a,\bar b)}{Q(z,\bar b)Q(a,\bar w)}.
\end{equation*}
The latter expression comes form (\ref{QPR}) together with the observation that the one variable polynomials
cancel in the four variable case.
The nature of $E_\rho (z,w,a,b)$ depends on the locations
of the points $z$, $w$, $a$, $b$, more precisely on the the values
of $\rho$ at these points. The main result is the following.


\begin{theorem}\label{thm:EmathcalE}
Let $(M_+,f)$ be a multi-sheeted algebraic domain. Then,
in the above notations,
$$
E_\rho (z,w; a,b)=\mathcal{E}_{f,f^*} (z, \bar{w}; a, \bar{b})\cdot
$$
$$
\cdot\left(\frac{\bar{z}-\bar{w}}{S(z)-\bar{w}}\right)^{\rho({z})}
\left(\frac{{w}-{z}}{\overline{S(w)}-{z}}\right)^{\rho({w})}
\left(\frac{\bar{a}-\bar{b}}{S(a)-\bar{b}}\right)^{\rho({a})}
\left(\frac{{b}-{a}}{\overline{S(b)}-{a}}\right)^{\rho({b})}\cdot
$$
$$
\cdot\left(\frac{S({z})-\bar{b}}{\bar{z}-\bar{b}}\right)^{\rho({z})}
\left(\frac{\overline{S({w})}-{a}}{{w}-{a}}\right)^{\rho({w})}
\left(\frac{S({a})-\bar{w}}{\bar{a}-\bar{w}}\right)^{\rho({a})}
\left(\frac{\overline{S({b})}-{z}}{{b}-{z}}\right)^{\rho({b})}.
$$
\end{theorem}


\section{Proof of the main result}\label{sec:proof}

Extending (\ref{dbarCauchy1}) to four variables gives
$$
\frac{\partial C_\rho (z,w;a,b)}{\partial \bar z}=\frac{\rho(z)}{\bar z-\bar w}-\frac{\rho(z)}{\bar z-\bar b},
$$
which tells that the function
$$
C_\rho(z,w;a,b)+\rho (z)\log\frac{\bar z-\bar b}{\bar z-\bar w}
$$
is analytic in $z$ in regions where $\rho(z)$ is constant, namely in each component of $\P\setminus f(\Gamma)$.
Hence so is the exponential of it, namely
$$
E_\rho (z,w; a,b)
\cdot(\frac{\bar{z}-\bar{b}}{\bar{z}-\bar{w}})^{\rho({z})}.
$$

Augmenting this we have that also the function
$$
F(z)=E_\rho (z,w; a,b)
\cdot\left(\frac{\bar{z}-\bar{b}}{\bar{z}-\bar{w}}\cdot\frac{S(z)-\bar w}{S(z)-\bar b}\right)^{\rho({z})}
$$
is analytic in $\P\setminus f(\Gamma)$, away from poles caused by
the presence of $S(z)$. Now we claim that it is even better than
that: $F(z)$ is meromorphic everywhere, hence is a rational
function. To prove this it is enough to prove that $F(z)$ is
continuous across $f(\Gamma)$. It is well-known that $E_\rho (z,w;
a,b)$ is continuous in $z$ (see for example
\cite{Gustafsson-Putinar98}), so we only have to bother about the
other factor. But it is easy to realize that also this is
continuous: spelling out as in (\ref{schwarzrho}) and assuming for
example that $\rho(z)$ increases by one unit as $f(\Gamma)$ is
crossed at a certain place we find that the factor
$$
\left(\frac{S(z)-\bar w}{\bar z-\bar w}\right)^{\rho({z})}
$$
changes from
$$
\frac{(f^*(\zeta_1)-\bar
w)\cdot\dots\cdot(f^*(\zeta_{\rho(z)})-\bar w)}{(\overline{f(\zeta_1)}-\bar
w)\cdot\dots\cdot(\overline{f(\zeta_{\rho(z)})}-\bar w)}
$$
to
$$
\frac{(f^*(\zeta_1)-\bar w)\cdot\dots\cdot(f^*(\zeta_{\rho(z)})-\bar
w)(f^*(\zeta_{\rho(z)+1})-\bar w)}{(\overline{f(\zeta_1)}-\bar
w)\cdot\dots\cdot(\overline{f(\zeta_{\rho(z)})}-\bar
w)(\overline{f(\zeta_{\rho(z)+1})}-\bar w)},
$$
which clearly is a continuous change since the new point
$\zeta_{\rho(z)+1}$ starts up on $\Gamma$. Similarly for the factor
$$
\left(\frac{\bar z-\bar b}{S(z)-\bar b}\right)^{\rho({z})}
$$

Repeating the above argument for $w$, $a$, $b$ it follows that the
function
$$
E_\rho (z,w; a,b) \cdot
\left(\frac{\bar{z}-\bar{b}}{\bar{z}-\bar{w}}
\cdot\frac{S(z)-\bar{w}}{S(z)-\bar{b}}\right)^{\rho({z})}
\cdot\left(\frac{{w}-{a}}{{{w}}-{z}}
\cdot\frac{\overline{S(w)}-{z}}{{\overline{S(w)}}-{a}}\right)^{\rho({w})}\cdot
$$
$$
\cdot\left(\frac{\bar{a}-\bar{w}}{\bar a-\bar{b}}
\cdot\frac{S(a)-\bar{b}}{S({a})-\bar{w}}\right)^{\rho({a})}
\cdot\left(\frac{{b}-{z}}{b-a}
\cdot\frac{\overline{S(b)}-{a}}{\overline{S({b})}-{z}}\right)^{\rho({b})}
$$
is rational in the variables $z$ $\bar w$, $a$, $\bar b$. Thus,
since also $\mathcal{E}_{f,f^*} (z, \bar{w}; a, \bar{b})$ is
rational in these variables, it is enough to prove that the formula
in the statement of the theorem holds just locally, somewhere. We
may then choose $z$, $w$, $a$, $b$ close to each other, so that in
particular $\rho(z)=\rho(w)=\rho(a)=\rho(b)$. In addition we may
assume that this value is the smallest value of $\rho$ occurring on
$\P$. The case that it is zero can be treated exactly as in the
proof of Theorem~6 in \cite{Gustafsson-Tkachev09}, which concerns
the special case $a=b=\infty$.

So let us for example assume that
$\rho(z)=\rho(w)=\rho(a)=\rho(b)=1$ (the general case will be
similar). Thus $f$ attains the values $z$, $w$, $a$, $b$ exactly
once in $M_+$, and these four points on the Riemann sphere are close
to each other. Let $\gamma$ be an arc in $\P$ from $b$ to $w$ (e.g.,
the geodesic arc). Then  the function $z\mapsto \log\frac{z-w}{z-b}$
has a single-valued branch, call it $\Log \frac{z-w}{z-b}$, in
$\P\setminus \gamma$, hence $\Log\frac{f-w}{f-z}$ is single-valued
in $M\setminus f^{-1}(\gamma)$. We may consider $f^{-1}(\gamma)$ as
a $1$-chain, and as such it has the same role for $\Log
\frac{z-w}{z-b}$ as $\sigma_f$ has for $\Log f$ in (\ref{dLogf}), so
that
$$
d\Log \frac{f-w}{f-b}=\frac{df}{f-w}-\frac{df}{f-b}-2\pi\I
dH_{f^{-1}(\gamma)}.
$$

If $f$ has degree $n$, then $f^{-1}(\gamma)$ consists of $n$ small
arcs, one of which is located on $M_+$. Let
$\sigma=f^{-1}(\gamma)\cap M_+$ be that arc and let
$\tilde{\sigma}=J(\sigma)$ be the reflected arc in $M_-$. In the
sequel we shall use $f^{-1}$ in the restricted sense
$f^{-1}=(f|_{M_+})^{-1}$. Thus
$$
\partial \sigma =f^{-1}(w)-f^{-1}(b)=(\frac{f-w}{f-b})|_{M_+},
$$
$$
\partial \tilde{\sigma} =\widetilde{f^{-1}(w)}-\widetilde{f^{-1}(b)}=(\frac{f^*-\bar w}{f^*-\bar b})|_{M_-}.
$$

Using (\ref{elimint}), (\ref{dLogf}) and (\ref{inttau}) we now get
$$
\mathcal{E}_{f,f^*}(z,\bar{w}; a,\bar b)
=\exp[\frac{1}{2\pi \I}\int_{M} (\frac{df}{f-z}-\frac{df}{f-a})\wedge d\,
\Log\frac{f^*-\bar{w}}{f^*-\bar b}]
$$
$$
=\exp[\frac{1}{2\pi \I}\int_{M_+} (\frac{df}{f-z}-\frac{df}{f-a})\wedge d\,
\Log \frac{f^*-\bar{w}}{f^*-\bar b}]\cdot\exp[\int_{\tilde{\sigma}}(\frac{df}{f-z}-\frac{df}{f-a})].
$$
Here we start by rewriting the last factor according to
$$
\exp\int_{\tilde{\sigma}}(\frac{df}{f-z}-\frac{df}{f-a})
=[\frac{f-z}{f-a}]_{\widetilde{f^{-1}(b)}}^{\widetilde{f^{-1}(w)}}
$$
$$
=\frac{\overline{f^*(f^{-1}(w))}-z}{\overline{f^*(f^{-1}(w))}-a}
\cdot\frac{\overline{f^*(f^{-1}(b))}-a}{\overline{f^*(f^{-1}(b))}-z}
=\frac{\overline{S(w))}-z}{\overline{S(w))}-a}\cdot \frac{\overline{S(w))}-a}{\overline{S(w))}-z}.
$$
Note that functions like $\overline{S(w)-z}=\overline{(S(w)-\bar
z)^{\rho(w)}}$ are single-valued in the present case (cf.
(\ref{schwarzrho})).

Next, the first factor can be integrated partially, to become
$$
\exp\frac{1}{2\pi \I}\int_{M_+} (\frac{df}{f-z}-\frac{df}{f-a})\wedge d\,
\Log\frac{f^*-\bar{w}}{f^*-\bar b}
$$
$$
=\exp[-\frac{1}{2\pi \I}\int_{\partial M_+} (\frac{df}{f-z}-\frac{df}{f-a})\,
\Log\frac{f^*-\bar{w}}{f^*-\bar b}]\cdot
$$
$$
\cdot\exp\frac{1}{2\pi \I}\int_{M_+} d\,(\frac{df}{f-z}-\frac{df}{f-a})\,
\Log\frac{f^*-\bar{w}}{f^*-\bar b}                 ]
$$
$$
=\exp[-\frac{1}{2\pi \I}\int_{\partial M_+} (\frac{df}{f-z}-\frac{df}{f-a})\,
\Log\frac{f^*-\bar{w}}{f^*-\bar b}]\cdot
$$
$$
\cdot \exp[\int_{M_+} (\delta_{f^{-1}(z)}-\delta_{f^{-1}(a)})dxdy\,
\Log\frac{f^*-\bar{w}}{f^*-\bar b}]
$$
$$
=\exp[-\frac{1}{2\pi \I}\int_{\partial M_+} (\frac{df}{f-z}-\frac{df}{f-a})\,
\Log\frac{f^*-\bar{w}}{f^*-\bar b}]
\cdot\frac{S(z)-\bar w}{S(z)-\bar b}
\cdot\frac{S(a)-\bar b}{S(a)-\bar w}.
$$

We have to rework all expressions which are not yet in the form appearing in the statement of the theorem.
So we next turn our attention to the first factor in the last obtained expression. This can be rewritten as
$$
\exp[-\frac{1}{2\pi \I}\int_{\partial M_+} (\frac{df}{f-z}-\frac{df}{f-a})\,
\Log\frac{f^*-\bar{w}}{f^*-\bar b}]
$$
$$
=\exp[-\frac{1}{2\pi \I}\int_{\partial M_+} (\frac{df}{f-z}-\frac{df}{f-a})\,
\Log\frac{\bar f-\bar{w}}{\bar f-\bar b}]
$$
$$
=\exp[-\frac{1}{2\pi \I}\int_{ M_+} d\,(\frac{df}{f-z}-\frac{df}{f-a})
\Log\frac{\bar f-\bar{w}}{\bar f-\bar b}]\cdot
$$
$$
\cdot\exp[\frac{1}{2\pi \I}\int_{ M_+} (\frac{df}{f-z}-\frac{df}{f-a})\,
d\,\Log\frac{\bar f-\bar{w}}{\bar f-\bar b}]
$$
$$
=\exp[-\int_{M_+} (\delta_{f^{-1}(z)}-\delta_{f^{-1}(a)})dxdy\,
\Log\frac{\bar f-\bar{w}}{\bar f-\bar b}]\cdot
$$
$$
\cdot\exp[\frac{1}{2\pi\I}\int_{M_+}
\biggl(\frac{df}{f -z}-\frac{df}{f -a}\biggr)\wedge
\biggl(\frac{d\bar f}{\bar f -\bar
w}-\frac{d\bar f}{\bar f - \bar b}\biggr)]\cdot
$$
$$
\cdot\exp\int_{M_+} (\frac{df}{f-z}-\frac{df}{f-a})d\,\overline{H_\sigma}
$$
$$
=\frac{\bar z-\bar b}{\bar z-\bar w}\cdot \frac{\bar a-\bar w}{\bar a-\bar b}
\cdot\exp[\frac{1}{2\pi\I}\int_{M_+}
\biggl(\frac{df}{f -z}-\frac{df}{f -a}\biggr)\wedge
\biggl(\frac{d\bar f}{\bar f -\bar
w}-\frac{d\bar f}{\bar f - \bar b}\biggr)]\cdot
$$
$$
\cdot\exp[-\int_{\sigma} (\frac{df}{f-z}-\frac{df}{f-a})]
$$
$$
=\frac{\bar z-\bar b}{\bar z-\bar w}\cdot\frac{\bar a-\bar w}{\bar a-\bar b}
 \cdot E_\rho(z,w;a,b)
\cdot \frac{w-a}{w-z} \cdot\frac{b-a}{b-z}.
$$

Now putting all the pieces together we obtain the formula in the
statement of the theorem.


\section{Examples}\label{sec:example}

\subsection{The ellipse}

Let $D$ denote the domain inside the ellipse
$$
\frac{x^2}{a^2}+\frac{y^2}{b^2}=1
$$
with semiaxes $a>b>0$ and foci $\pm c=\pm \sqrt{a^2-b^2}$. The
exterior domain $\P\setminus \overline{D}$ is known to be a null
quadrature domain \cite{Sakai81} for the Euclidean metric. For the
spherical metric it is a two point quadrature domain. Indeed, the
Schwarz function for the ellipse is
$$
S(z)=\frac{a^2+b^2}{c^2} z\pm \frac{2ab}{c}\sqrt{z^2-c^2},
$$
and for $h$ holomorphic in $\P\setminus\overline{D}$ and smooth up
to the boundary we have
$$
\frac{1}{2\pi\I} \int_{\P\setminus\overline{D}} h(z) \frac{d\bar
z\wedge dz}{(1+|z|^2)^2} =\sum\res_{z\in \P\setminus\overline{D}}
h(z) \frac{S(z)dz}{1+zS(z)}.
$$
The residues come from the zeros of $1+zS(z)$ in the exterior of the
ellipse, and straight-forward computations show that there are
exactly two such zeros, located on the imaginary axis and more
precisely given by
$$
z=\pm z_0=\pm\frac{1}{\I c}\sqrt{a^2+b^2+2a^2b^2 +2ab\sqrt{1+
a^2+b^2+a^2b^2}}.
$$
Thus we have a quadrature identity of the form
\begin{equation}\label{qiellipse}
\frac{1}{2\pi\I} \int_{\P\setminus\overline{D}} h(z) \frac{d\bar
z\wedge dz}{(1+|z|^2)^2}=c_0 (h(z_0)+h(-z_0)),
\end{equation}
$c_0$ being the residue of $\frac{S(z)dz}{1+zS(z)}$ at $z=\pm z_0$.

The exterior of the ellipse is the conformal image of the unit disk
under the Joukowski map
$$
f(\zeta)=\frac{c^2\zeta^2+(a+b)^2}{2(a+b)\zeta}.
$$
Thus $({\D},f)$, or simply $(\P\setminus\overline{D},z)$ with $z$
denoting the identity function, is a (single-sheeted) algebraic
domain. The same function $f$ maps the exterior of the unit disk
onto a multi-sheeted algebraic domain, i.e.,
$(\P\setminus\overline{\D},f)$ is (or represents) such  a domain. It
covers $D$ twice and $\P\setminus \overline{D}$ once, in other words
the counting function is
$$
\rho(z)=\begin{cases}
      2 \quad {\rm for\,\,} z\in D,\\
      1  \quad {\rm for\,\,} z\in \P\setminus \overline{D}.
     \end{cases}
$$
Since there are several sheets the associated quadrature identity is
best expressed in a form pulled-back to $\P\setminus\overline{\D}$,
i.e., on the form (\ref{qi}). A slightly weaker form is
$$
\frac{1}{2\pi\I} \int_{\P} \rho(z)h(z) \frac{d\bar z\wedge
dz}{(1+|z|^2)^2}=c_1 (h(z_1)+h(-z_1)),
$$
where
$$
\pm z_1=\pm\frac{1}{\I c}\sqrt{a^2+b^2+2a^2b^2 -2ab\sqrt{1+
a^2+b^2+a^2b^2}}
$$
and $c_1$ is the residue of $\frac{S(z)dz}{1+zS(z)}$ at $z=\pm z_1$.
This form is weaker because it only uses test functions $h(z)$ that
take the same values on the two sheets over $D$.


\subsection{Neumann's oval: classification of level curves}

By inversion in the unit circle and a rotation by $90$ degrees (for
convenience) the ellipse transforms into a curve known as Neumann's
oval \cite{Neumann07}, \cite{Neumann08}, \cite{Shapiro92},
\cite{Langer-Singer07} with equation
\begin{equation}\label{neumann}
a^2b^2 (x^2+y^2)^2 -a^2x^2 -b^2y^2 =0.
\end{equation}
The exterior of the ellipse transforms into a bounded domain
$\Omega$, and since inversions and rotations are rigid
transformations with respect to the spherical measure, also $\Omega$
will be a two point quadrature domain for the spherical measure. The
formula is immediately obtained by inversion and rotation of
(\ref{qiellipse}). The domain $\Omega$ also satisfies a quadrature
identity for the Euclidean measure (indeed, both types of quadrature
identities are equivalent to $S(z)$ being meromorphic in $\Omega$).
The latter quadrature identity is somewhat simpler, namely
$$
\frac{1}{2\pi\I}\int_\Omega hd\bar
zdz=\frac{a^2+b^2}{4a^2b^2}(h(-\frac{c}{2ab})+h(\frac{c}{2ab})).
$$

It turns out to be a quite rewarding task to investigate the
algebraic curves corresponding to the level curves of the left
member in (\ref{neumann}), and in particular to determine which of
them correspond to multi-sheeted algebraic domains. Most types of
phenomena which could possibly show up really do show up among these
curves. This task is what we are going to undertake for the
remainder of this section.

In order to simplify a little we first scale so that the quadrature
nodes above become $\pm 1$. This means that $2ab=c$. Then we need
only one parameter (in place of the two, $a$ and $b$), which we take
to be
$$
r=\frac{\sqrt{a^2+b^2}}{c}>1.
$$
The quadrature identity now becomes
\begin{equation}\label{twopointqi}
\int_\Omega h dxdy =\pi r^2 (h(-1)+h(1)),
\end{equation}
holding for all integrable analytic functions $h$ in $\Omega$. Set
$$
Q(z,w)=z^2 w^2-z^2-w^2 -2r^2 zw,
$$
which is the polarized version ($(z,\bar z)$ polarizes into $(z,w)$)
of the left member in (\ref{neumann}). There are exactly two open
sets for which the quadrature identity (\ref{twopointqi}) holds,
namely
$$
\Omega = \{z\in \C: Q(z,\bar z)<0\}
$$
and
$$
[\Omega] = \{z\in \C: Q(z,\bar z)<0\}\cup \{0\}.
$$
The latter is just the completion of the former with respect to one
missing point. The domain $\Omega$ (or $[\Omega]$) can be viewed as
two disks glued together, or ``smashed'',  or ``added'', and has
been studied by many authors, for example \cite{Crowdy-Marshall04},
\cite{Langer-Singer07}, \cite{Levine-Peres10}.

To analyse the level curves of $Q(z,\bar z)$ we set, for any
$\alpha\in \R$,
\begin{equation}\label{Qalpha}
Q_\alpha(z,w)=z^2 w^2-z^2-w^2 -2r^2 zw -\alpha.
\end{equation}
Let
$$
Q_\alpha(t,z,w)=z^2 w^2-z^2t^2-w^2 t^2-2r^2 zw t^2 -\alpha t^4
$$
be the corresponding homogenous polynomial. We shall keep $r>1$
fixed and just vary $\alpha$.
The real locus of $Q_\alpha$ in $\mathbb{C}$ is
$$
{\rm loc}_\R \,Q_\alpha=\{z\in\C:  Q(z,\bar z)=\alpha\}.
$$
It represents the intersection with $\{w=\bar z\}$ (``the real'') of
the complex locus in $\C^2$,
$$
{\rm loc}_{\C}\,Q_\alpha= \{(z,w)\in\C^2: Q(z,w)=\alpha\},
$$
which has a natural completion in the projective space $\P_2(\C)$
as
$$
{\rm loc\,}Q_\alpha={\rm loc}_{\P_2(\C)} \,Q_\alpha
= \{(t:z:w)\in\P_2(\C): Q_\alpha (t,z,w)=0\}.
$$
Here $\C^2$ is embedded in $\P_2(\C)$ so that $(z,w)$ corresponds to
$(1:z:w)$. By real points in $\C^2$ we mean points $(z,w)$
satisfying $w=\bar z$, and these are the fixed points of the
involution $J:(z,w)\mapsto (\bar w,\bar z)$. In projective
coordinates the involution is $J:(t:z:w)\mapsto (\bar t:\bar w:\bar z)$.
In addition to the anticonformal involution $J$,
the curve ${\rm loc}_{\C}\,Q_\alpha$ has the conformal symmetries $(z,w)\mapsto (w,z)$ and
$(z,w)\mapsto (-z,-w)$.

Let $(M_\alpha,J_\alpha)$ be the compact symmetric Riemann surface
corresponding to $({\rm loc\,}Q_\alpha, J)$. As point sets they are
identical except for a few singular points on ${\rm loc\,}Q_\alpha$,
which are resolved on $M_\alpha$. The real locus of $Q_\alpha$
corresponds to the fixed point set $\Gamma_\alpha$ of $J_\alpha$.

A first observation is that the function $z\mapsto Q(z,\bar z)$ has
five stationary points in the complex plane. There are two global
minima, on the level $\alpha=-(r^2+1)^2$, there are two saddle
points on the level $\alpha=-(r^2-1)^2$ and there is one local
maximum (at the origin), on the level $\alpha=0$. These three values
of $\alpha$ will correspond to changes of regime for the algebraic
curve $Q_\alpha(z,w)=0$.

Solving the equation $Q_\alpha(z,w)=0$ for $w$ as a function of $z$
gives the Schwarz functions for the curves in the real locus. The
result is
\begin{equation}\label{w(z)}
w=S_\alpha (z)=\frac{1}{z^2-1}(r^2 z\pm\sqrt{z^4+(r^4-1+\alpha)z^2-\alpha}).
\end{equation}
We see that $S_\alpha(z)$ in general has four branch points. The levels $\alpha$ at which
$Q(z,\bar z)$ has stationary points are exactly those values of $\alpha$
for which some or all of these branch points resolve:
for $\alpha=-(r^2\pm 1)^2$ the square root resolve completely into second order polynomials,
and for $\alpha=0$ one pair of branch points resolves.

Let $p=p(\alpha)$ denote the genus of $M_\alpha$, or equivalently of ${\rm
loc\,}Q_\alpha$. The degree of $Q_\alpha$ is four, hence the genus
formula in algebraic geometry \cite{Griffith-Harris}, \cite{Namba84}
tells that
$$
p+ s=\frac{(4-1)\cdot (4-2)}{2}=3,
$$
where $s\geq 0$ is a certain number related to the singular points.
An analysis, carried out in detail in \cite{Gustafsson88}, shows
that ${\rm loc\,}Q_\alpha$ passes through the points $(0:1:0)$ and
$(0:0:1)$ and that it at each of these points has two simple cusps
of multiplicity one with distinct tangent directions. In particular,
the points $(0:1:0)$ and $(0:0:1)$ are singular, and it turns out
that each of them gives the contribution $+1$ to $s$. Except for the
above two points of infinity, ${\rm loc\,}Q_\alpha$ stays in $\C^2$.
Thus, what remains of the genus formula is
\begin{equation}\label{genusformula}
p+e=1,
\end{equation}
where $e$ denotes the contribution to $s$ which comes from finite
singular points. By (\ref{genusformula}), $e\leq 1$, so there is at
most one finite singular point, and this must be visible in the real
because nonreal singular points necessarily come in pairs.

The above analysis preassumes that the curve ${\rm loc\,}Q_\alpha$,
or polynomial $Q_\alpha$, is irreducible. This is the case for most
values of $\alpha$, but there are two exceptions:

$(i)$ For $\alpha=-(r^2-1)^2$, $Q_\alpha$ factors as
$$
Q_\alpha(z,w)=((z+1)(w+1)-r^2)((z-1)(w-1)-r^2).
$$
Each factor defines its own algebraic curve and Riemann surface.
These have genus zero and are moreover symmetric: the zero locus of
each factor is preserved under the involution $(z,w)\mapsto (\bar
w,\bar z)$. The real locus ${\rm loc}_\R \,Q_\alpha$ is the union of
two intersecting circles, those of radius $r$ and centers $\pm 1$.
The intersection points, $z=\pm\sqrt{r^2-1}$ are saddle points for the function
$z\mapsto Q(z,\bar z)$.

$(ii)$ For $\alpha=-(r^2+1)^2$, $Q_\alpha$ factors as
$$
Q_\alpha(z,w)=((z+1)(w-1)-r^2)((z-1)(w+1)-r^2),
$$
where again each factor defines its own algebraic curve and Riemann
surface of genus zero. However, in the present case they are not
symmetric, instead the involution maps each of these Riemann
surfaces onto the other. The real locus consists only of the two
points $\pm r$. This can easily be understood by observing that the
value $\alpha=-(r^2+1)^2$ is the global infimum of $Q(z,\bar z)$.

In both of the reducible cases Bezout's theorem says that there
should be four points of intersection between the two curves (since
these have degree two). These intersection points are the two points
of infinity $(0:1:0)$ and $(0:0:1)$ plus, in the first case the
intersection points of the two circles (in the real), and in the
second case the two points $\pm r$.

Besides the above two special values of $\alpha$, also the
quadrature value $\alpha=0$ is exceptional. This is a local maximum
value for $Q(z,\bar z)$. The local maximum is attained at $z=0$, and
the corresponding point $(0,0)$ on the algebraic curve is a singular
point (since both partial derivatives of $Q$ vanish there). Thus
$e=1$ in the genus formula (\ref{genusformula}), hence $p=0$.

We now embark the full classification. Pictures for the case
$r=\sqrt{2}$ with $\alpha=3, -0.5,-1,-1.5$ are shown in
figures~\ref{picture1} and \ref{picture2}, where the shaded areas
are the sets where $Q_\alpha (z,\bar z)<0$.

\begin{itemize}

\item For $\alpha>0$,
$\Gamma_\alpha$ has one component and there are no singular points
visible in the real. Therefore, by (\ref{genusformula}) and the
remark following it $M_\alpha$ has genus one. This means that the
symmetry line $\Gamma_\alpha$ is not able to separate $M_\alpha$
into two halves (see more precisely discussions in Section~2.2 in
\cite{Schiffer-Spencer54}). Thus $M_\alpha\setminus \Gamma_\alpha$
has only one component, and it will not generate any algebraic
domain (even multi-sheeted), despite the nice picture in the real,
with a smooth algebraic curve bounding a simply connected region
(figure~\ref{picture1}, left). $M_\alpha$ can be viewed as the
double of a M\"obius band.

\item At $\alpha=0$ the genus of $M_\alpha$ collapses to zero, $\Gamma_\alpha$
has still only one component even though ${\rm loc}_\R \,Q_\alpha$
has gotten an additional point, $(0,0)$, which is a singular point
of ${\rm loc\,}Q_\alpha$ (resolved on $M_\alpha$). It follows that
$M_\alpha\setminus \Gamma_\alpha$ has two components, one of which,
say $M_+$, is mapped conformally onto the quadrature domain
$[\Omega]$ by the analytic function $f$ which corresponds to the
projection $(z,w)\mapsto z$ on ${\rm loc\,}Q_\alpha$.

\item For $-(r^2-1)^2<\alpha<0$ the genus of $M_\alpha$ is again one, and the
singular point $(0,0)$ in the previous case has now grown up to a
curve, hence $\Gamma_\alpha$ has two components. Also
$M_\alpha\setminus \Gamma_\alpha$ has two components, say $M_{\pm}$,
and $M_\alpha$ can be viewed as the double of $M_+$ (or $M_-$),
which topologically is an annulus.

The meromorphic function $f:M_\alpha\to \P$ which corresponds to
$(z,w)\mapsto z$ on ${\rm loc\,}Q_\alpha$ is however no longer
univalent (not even locally univalent) on what corresponds to $M_+$
in the previous case. Therefore this gives a now only multi-sheeted
algebraic domain, with $f(M_+)$ consisting of a main piece which
contains the points $\pm 1$ and the origin (the dashed area in
figure~\ref{picture1}, right) plus a smaller piece around the origin
(the bounded undashed region). The latter piece thus is covered
twice, and the two sheets are connected via two branch points of the
Schwarz function.

\item For $\alpha=-(r^2-1)^2$ the curve is reducible, hence $M_\alpha$ is the
union of two independent Riemann surfaces, both of genus zero and
symmetric under $J_\alpha$. In the real locus ${\rm loc}_\R
\,Q_\alpha$ we simply have two intersecting circles, those centered
at $\pm 1$ and having radius $r$ (figure~\ref{picture2}, left).
Explicitly:
$$
Q_\alpha (z,\bar z)= (|z-1|^2-r^2)(|z+1|^2-r^2).
$$
$M_\alpha$ is the union of two Riemann spheres and can be viewed as
the double of two disks.

\item For $-(r^2+1)^2<\alpha<-(r^2-1)^2$ we are back to the case of genus
one with $\Gamma_\alpha$ and $M_\alpha\setminus \Gamma_\alpha$ both
having two components. However, the situation has changed in the
sense that the involution goes the other way (like $z\mapsto -\bar
z$ in place of $z\mapsto \bar z$ in a right-angled period
parallelogram), it may be more natural in this case to think of
$M_\alpha$ as the double of a cylinder than as the double of an
annulus (even though these two types of domains are topologically
equivalent). The real locus consists of two closed curves, one
enclosing two branch points close to $z=1$, the other enclosing two
branch points close to $z=-1$, and with $f:M_\alpha\to \P$ as
before, $f$ maps $M_+$ (say) onto the dashed region to the right in
figure~\ref{picture2} (right) covered twice and the (unbounded)
undashed region covered once.

\item When $\alpha=-(r^2+1)^2$ the two closed curves in the real locus
of the previous case have shrunk to two points, the minimum points $z=\pm 1$ of
$Q(z,\bar z)$. The curve is reducible and $M_\alpha$ hence is
the disjoint union of two Riemann surfaces (of genus zero), but
these are not symmetric under $J$. Instead the involution $J$ maps each of
them onto the other. The algebraic curve ${\rm loc}_\R \,Q_\alpha$
consists of two pieces, which meet each other in two points of
tangency. This is what is seen in the real locus.
Thus $M_\alpha$ can be viewed as the double of the Riemann sphere, and $J_\alpha$
has no fixed points.

\item When $\alpha<-(r^2+1)^2$ finally, there is no real locus at all.
The genus is one, but now $J$ has no fixed points at all
($\Gamma_\alpha$ is empty). Therefore no algebraic domain (even
multisheeted) can be associated to this case. $M_\alpha$ can be
viewed as the double of Klein's bottle (see again
\cite{Alling-Greenleaf71}, and also \cite{Schiffer-Spencer54} for
doubles of nonorientable surfaces).

\end{itemize}

We summarize the discussions as follows.

\begin{proposition}

The surface $M_\alpha$ (possibly disconnected) can be viewed as the
double of the Klein surface $N_\alpha=M_\alpha/J_\alpha$, which in
the different regimes of $\alpha\in\R$ is of the following
topological type.

\begin{itemize}

\item For $\alpha>0$: a M\"obius band

\item For $\alpha=0$: a disk.

\item For $-(r^2-1)<\alpha<0$: an annulus.

\item For $\alpha=-(r^2-1)^2$: two disjoint disks.

\item For $-(r^2+1)^2<\alpha<-(r^2-1)^2$: a cylinder.

\item For $\alpha=-(r^2+1)^2$: a sphere.

\item For $\alpha<-(r^2+1)^2$: a Klein's bottle.

\end{itemize}

The pair $(M_+, f)$ defines a single-sheeted algebraic domain
for $\alpha=0$ and a multi-sheeted algebraic domain for $-(r^2-1)^2<\alpha<0$ and $-(r^2+1)^2<\alpha<-(r^2-1)^2$.
For $\alpha=-(r^2-1)^2$ it defines two single-sheeted algebraic domains (which intersect in the complex plane).

The exponential transform of the above multi-sheeted algebraic
domains is given by Theorem~\ref{thm:EmathcalE}, where the
elimination function is
$$
\mathcal{E}_{f,f^*} (z, \bar w; a,\bar b)=\frac{Q_\alpha(z,\bar w)Q_\alpha(a,\bar b)}{Q_\alpha(z,\bar b)Q_\alpha(a,\bar w)}
$$
with $Q_\alpha$ as in (\ref{Qalpha}) and the Schwarz function is given by (\ref{w(z)}).

\end{proposition}


\newpage
\begin{figure}[h]
\centering
\includegraphics[width=0.4\textwidth]{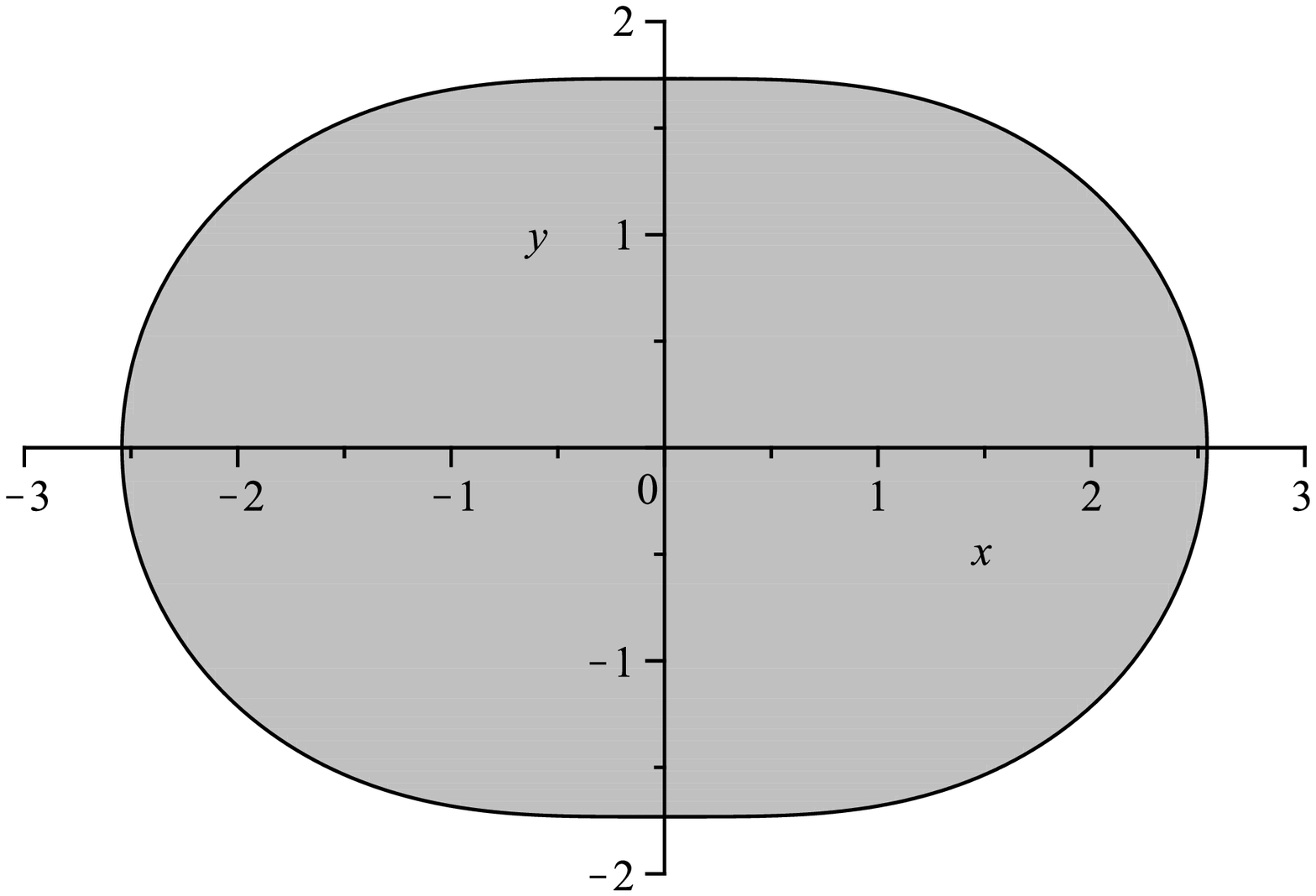}
\includegraphics[width=0.4\textwidth]{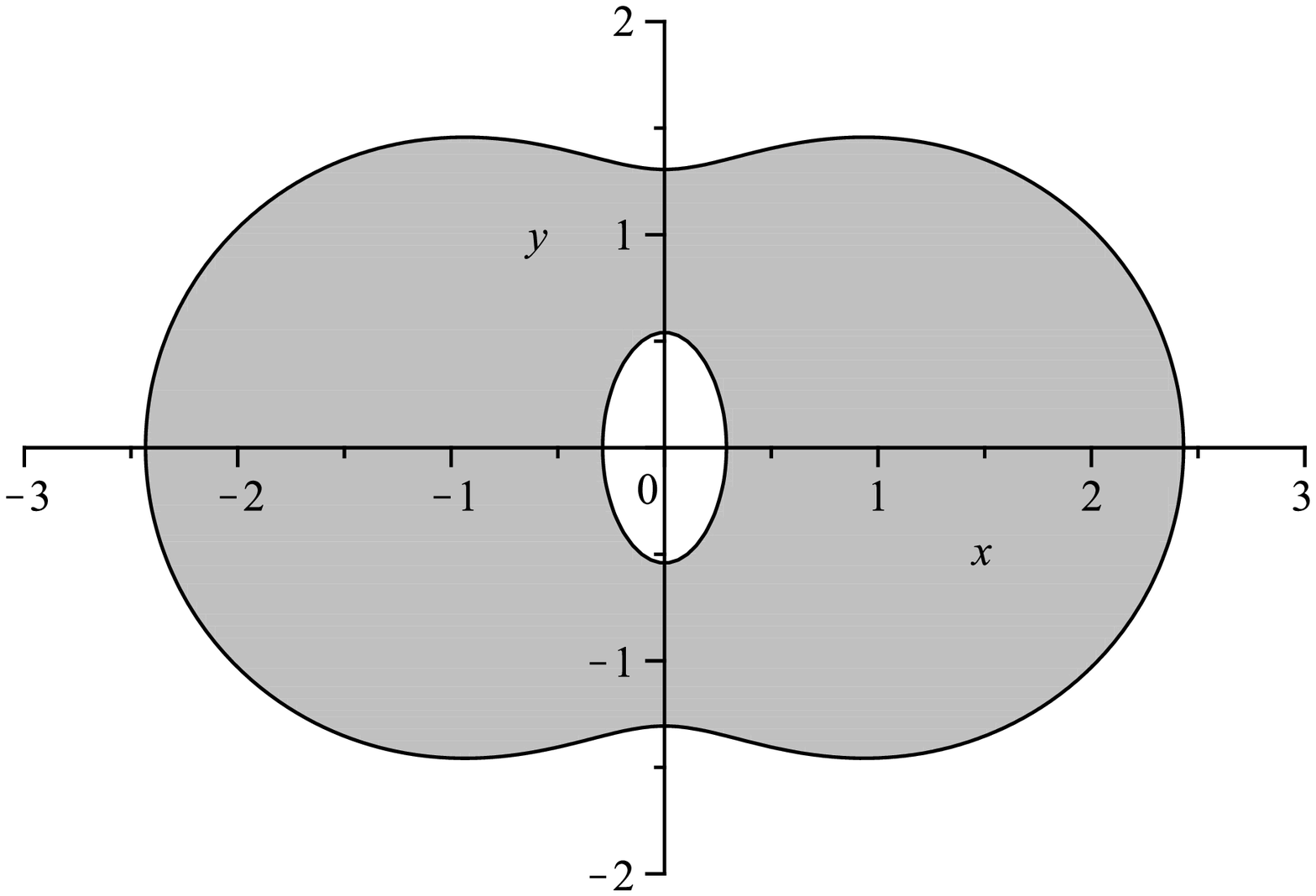}

 \caption{The level set $\Gamma_\alpha$ for $r=\sqrt{2}$ and for $\alpha=3$ and $\alpha=-0.5$}
 \label{picture1}
\end{figure}

\begin{figure}[h]
\centering
\includegraphics[width=0.4\textwidth]{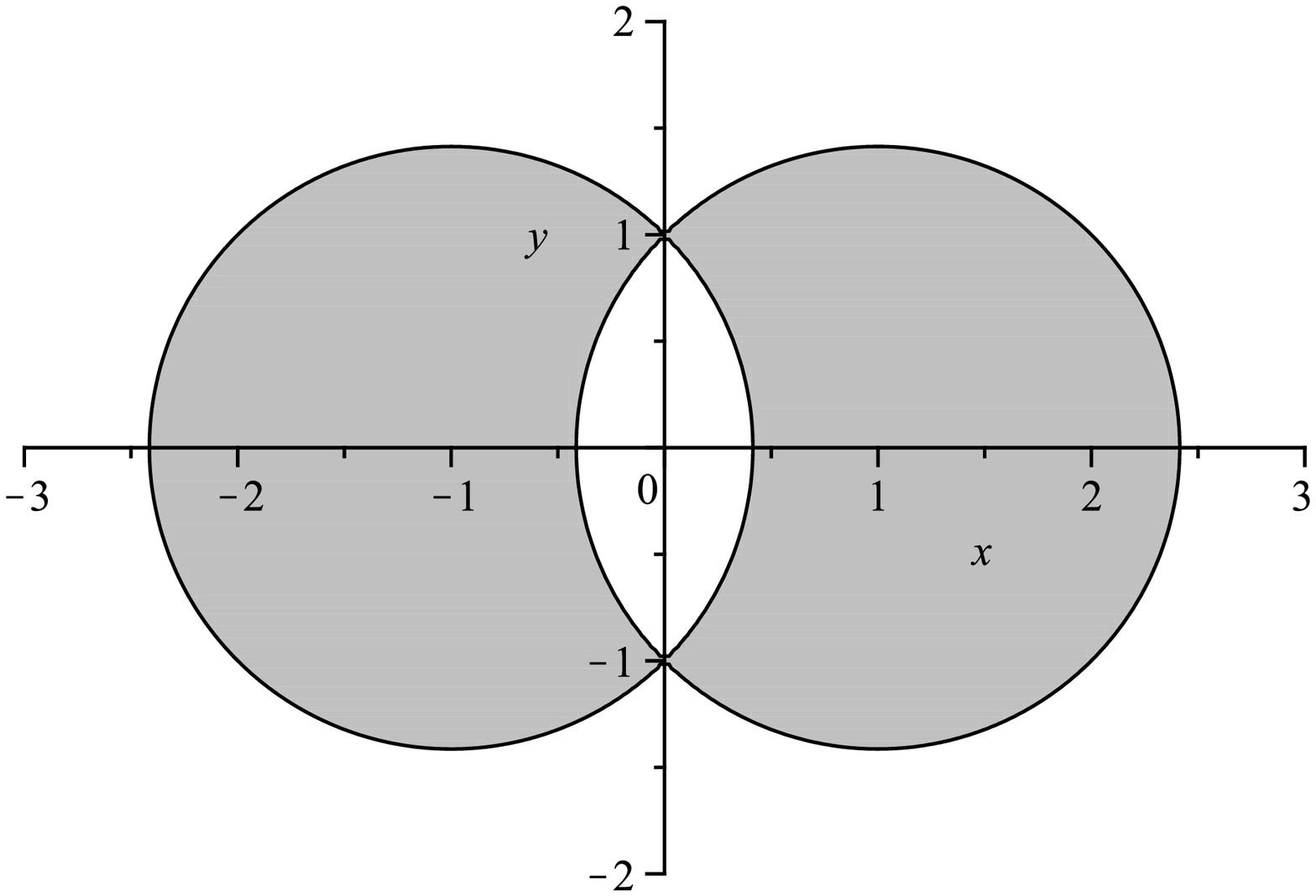}
\includegraphics[width=0.4\textwidth]{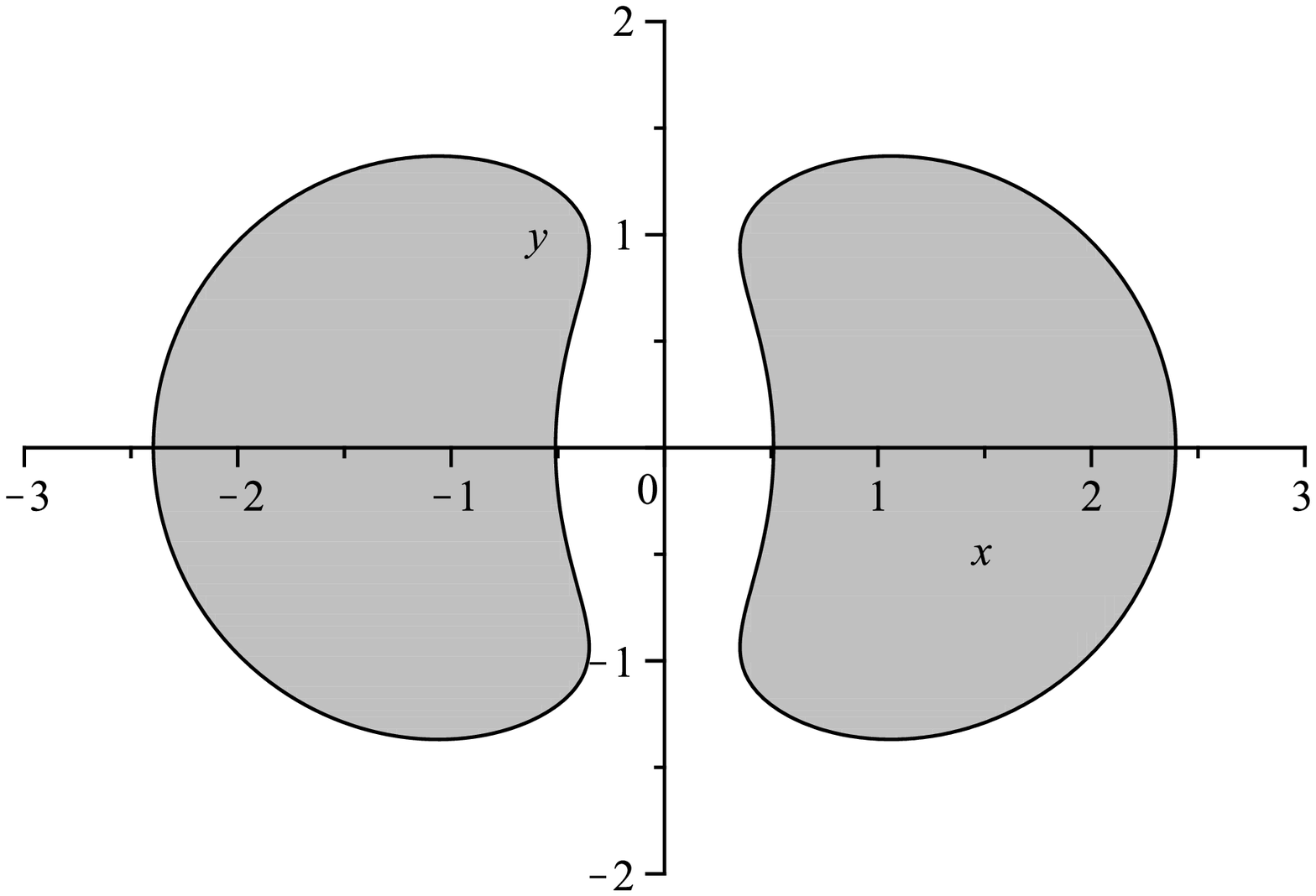}

 \caption{The level set $\Gamma_\alpha$ for $\alpha=-1$ and $\alpha=-1.5$}
 \label{picture2}
\end{figure}


\end{document}